\documentclass[11pt,reqno]{amsart}
\usepackage{amsmath,amssymb,mathrsfs,hyperref}
\title[$L^p$ estimates for B-operator]{$L^p$ mean estimates for an operator preserving inequalities between polynomials}
  \newtheorem{lemma}{Lemma}
   \newtheorem{theorem}{Theorem}
   \newtheorem{corollary}{Corollary}
   
   \newtheorem{remark}{Remark}
 \numberwithin{equation}{section}
 \newtheorem{thmx}{Theorem}[section]
 
\begin{document}
\maketitle
 \date{}
 \begin{center}
 \author{N. A. Rather$^1$ and Suhail Gulzar$^{2,}$\footnote{The work is supported by Council of Scientific and Industrial Research, New Delhi, under grant F.No. 09/251(0047)/2012-EMR-I.}}\\
  \address{
  $^{1,2}$Department of mathematics,\\
   University of Kashmir,\\
   Harzarbal, Sringar 190006, India,\\
   email: \texttt{sgmattoo@gmail.com.}
  }
 \end{center}

 \begin{abstract}
If $P(z)$ be a polynomial of degree at most $n$ which does not vanish in $|z| < 1$, it was recently formulated by Shah and Liman \cite[\textit{Integral estimates for the family of $B$-operators, Operators and Matrices,} \textbf{5}(2011), 79 - 87]{wl} that for every $R\geq 1$, $p\geq 1$,
 \[\left\|B[P\circ\sigma](z)\right\|_p \leq\frac{R^{n}|\Lambda_n|+|\lambda_{0}|}{\left\|1+z\right\|_p}\left\|P(z)\right\|_p,\] where $B$ is a $ \mathcal{B}_{n}$-operator with parameters $\lambda_{0}, \lambda_{1}, \lambda_{2}$ in the sense of Rahman \cite{qir}, $\sigma(z)=Rz$ and $\Lambda_n=\lambda_{0}+\lambda_{1}\frac{n^{2}}{2} +\lambda_{2}\frac{n^{3}(n-1)}{8}$. Unfortunately the proof of this result is not correct. In this paper, we present a more general sharp $L_p$-inequalities for $\mathcal{B}_{n}$-operators which not only provide a correct proof of the above inequality as a special case but also  extend them for $ 0 \leq p <1$ as well. 
 \end{abstract}
 \begin{center}
 \section{\bf{ Introduction and statement of results}}
\end{center}
\hspace{4mm} Let $\mathscr{P}_n $ denote the space of all complex polynomials $P(z)=\sum_{j=0}^{n}a_{j}{z}^{j}$ of degree at most $n$.  For $P\in\mathscr{ P}_{n}$, define

$$\left\|P(z)\right\|_{0}:=\exp\left\{\frac{1}{2\pi}\int_{0}^{2\pi}\log\left| P(e^{i\theta})\right|d\theta\right\},$$\\
$$\left\|P(z)\right\|_{p}:=\left\{\frac{1}{2\pi}\int_{0}^{2\pi}\left| P(e^{i\theta})\right|^{p}\right\}^{1/p},\,\, 0< p<\infty$$ \\
$$\left\|P(z)\right\|_{\infty}:=\underset{\left|z\right|=1}{Max}\left|P(z)\right|,$$

and denote for any complex function $\sigma : \mathbb C\rightarrow \mathbb C$ the composite function of  
$P$ and $\sigma$, defined by
$\left(P\circ\sigma\right)(z):=P\left(\sigma(z)\right)\,\,\,\,(z\in\mathbb C)$, as $P\circ\sigma$.\\

\indent A famous result known as Bernstein's inequality (for reference, see \cite[p.531]{9}, \cite[p.508]{rs} or \cite{11} states that if $P\in \mathscr{P}_n$, then
\begin{equation}\label{eq1}
\left|P^{\prime}(z)\right|_\infty\leq  n\left\|P(z)\right\|_\infty,
\end{equation} 
whereas concerning the maximum modulus of $P(z)$ on the circle $\left|z\right|=R>1$, we have
\begin{equation}\label{eq2}
\left\|P(Rz)\right\|_\infty\leq  R^{n}\left\|P(z)\right\|_\infty,\,\,\, R\geq 1,
\end{equation}
(for reference, see \cite[p.442]{8} or \cite[vol.I, p.137]{9} ).\\
Inequalities \eqref{eq1} and \eqref{eq2} can be obtained by letting $p\rightarrow\infty$ in the inequalities
\begin{equation}\label{eq3}
\left\|P^{\prime}(z)\right\|_{p}\leq n\left\|P(z)\right\|_{p},p\geq 1
\end{equation} 
and
\begin{equation}\label{eq4}
\left\|P(Rz)\right\|_{p}\leq R^{n}\left\|P(z)\right\|_{p},\,\,\,\,R>1,\,\,\,\,\,p>0,
\end{equation}
respectively. Inequality \eqref{eq3} was found by Zygmund \cite{z} whereas inequality \eqref{eq4} is a simple consequence of a result of Hardy \cite{h} (see also \cite[Th. 5.5]{rs83}). Since inequality \eqref{eq3} was deduced from M. Riesz's interpolation formula \cite{r} by means of Minkowski's inequality, it was not clear, whether the restriction on $p$ was indeed essential. This question was open for a long time. Finally Arestov \cite{va} proved that \eqref{eq3} remains true for $0<p<1$ as well. 

 \indent If we restrict ourselves to the class of polynomials $P\in \mathscr{P}_n$ having no zero in $|z|<1$, then inequalities \eqref{eq1} and \eqref{eq2} can be respectively replaced by
\begin{equation}\label{eq5}
\left\|P^{\prime}(z)\right\|_\infty\leq \dfrac{n}{2}\left\|P(z)\right\|_\infty,
\end{equation}
and
\begin{equation}\label{eq6}
\left\|P(Rz)\right\|_\infty\leq \frac{R^{n}+1}{2}\left\|P(z)\right\|_\infty\,\,\,\,\,\,R>1.
\end{equation}
Inequality \eqref{eq5} was conjectured by Erd\"{o}s and later verified by Lax \cite{7}, whereas inequality \eqref{eq6} is due to Ankey and Ravilin \cite{1}.\\
Both the inequalities \eqref{eq5} and \eqref{eq6} can be obtain by letting $p\rightarrow\infty$ in the inequalities
\begin{equation}\label{eq7}
\left\|P^{\prime}(z)\right\|_{p}\leq n\frac{\left\|P(z)\right\|_{p}}{\left\|1+z\right\|_{p}},\,\,\,\,\,p\geq 0
\end{equation} 
and
\begin{equation}\label{eq8}
\left\|P(Rz)\right\|_{p}\leq \frac{\left\|R^{n}z+1\right\|_{p}}{\left\|1+z\right\|_{p}}\left\|P(z)\right\|_{p}\,\,\,\,\,\,R>1,\,\,\,\,\, p>0.
\end{equation}
Inequality \eqref{eq7} is due to De-Bruijn \cite{nb} for $p\geq 1$. Rahman and Schmeisser \cite{rs88} extended it for $0\leq p<1$ whereas the inequality \eqref{eq8} was proved by Boas and Rahman \cite{br} for $p\geq 1$ and later it was extended for $0\leq p<1$ by Rahman and Schmeisser \cite{rs88}.

\indent Q. I. Rahman \cite{qir} (see also Rahman and Schmeisser \cite[p. 538]{rs}) introduced a class $\mathcal{B}_n$ of operators $B$ that carries a polynomial $P\in\mathscr{P}_n$ into 
  \begin{equation}\label{BO}
 B[P](z):=\lambda_0P(z)+\lambda_1\left(\dfrac{nz}{2}\right)\dfrac{P^{\prime}(z)}{1!}+\lambda_2\left(\dfrac{nz}{2}\right)^2\dfrac{P^{\prime\prime}(z)}{2!},
  \end{equation}\label{BO'}
  where $\lambda_0,\lambda_1$ and $\lambda_2$ are such that all the zeros of
  \begin{equation}\label{uz}
  U(z):=\lambda_0+\lambda_1C(n,1)z+\lambda_2C(n,2)z^2\,\,\,\,\,
  \end{equation} 
  where
  $C(n,r)=\dfrac{n!}{r!(n-r)!}\,\,\,\,\,0\leq r\leq n,$
  lie in half plane $|z|\leq\left|z-n/2\right|.$\\
\indent As a generalization of inequality \eqref{eq1} and \eqref{eq5}, Q. I. Rahman \cite[inequality 5.2 and 5.3]{qir} proved that if $P\in\mathscr{P}_n,$ and $B\in \mathcal{B}_n$ then
  \begin{equation}\label{eq9}
 | B[P](z)|\leq |\Lambda_n|\|P(z)\|_\infty,\,\,\,\,\,\,\textnormal{for}\,\,\,\,\,\,\,|z|\geq 1,
  \end{equation}
  and if $P\in\mathscr{P}_n,$ $P(z)\neq 0$ in $|z|<1,$ then
  \begin{equation}\label{eq10}
 | B[P](z)|\leq\dfrac{1}{2} \left\{|\Lambda_n|+|\lambda_0|\right\}\|P(z)\|_\infty,\,\,\,\,\,\,\textnormal{for}\,\,\,\,\,\,\,|z|\geq 1,
  \end{equation}
  where \begin{equation}\label{phi}
\Lambda_n=\lambda_{0}+\lambda_{1}\frac{n^{2}}{2} +\lambda_{2}\frac{n^{3}(n-1)}{8}.
  \end{equation}
  \indent As a corresponding generalization of inequalities \eqref{eq2} and \eqref{eq4}, Rahman and Schmeisser \cite[p. 538]{rs} proved that if $P\in\mathscr{P}_n,$ then
  \begin{equation}\label{eq11}
  \left|B[P\circ\sigma](z)\right|\leq R^{n}|\Lambda_n|\left\|P(z)\right\|_\infty \,\,\,\mbox{} \textrm{for}\,\mbox{}\,\,\,|z|= 1.
  \end{equation}
  and if $P\in\mathscr{P}_n,$ $P(z)\neq 0$ in $|z|<1,$ then as a special case of Corollary 14.5.6 in \cite[p. 539]{rs}, we have
  \begin{equation}\label{eq12}
  \left|B[P\circ\sigma](z)\right|\leq \frac{1}{2}\left\{R^{n}|\Lambda_n|+|\lambda_{0}|\right\}\left\|P(z)\right\|_{\infty} \,\,\,\mbox{} \textrm{for}\,\mbox{}\,\,\,|z|=1,
  \end{equation}
  where $\sigma(z):=Rz,\,\,R\geq 1$ and $\Lambda_n$ is defined by \eqref{phi}.\\
  \indent Inequality \eqref{eq12} also follows by combining the inequalities (5.2) and (5.3) due to Rahman \cite{qir}.\\
  
  \indent As an extension of inequality \eqref{eq11} to $L_p$-norm, recently Shah and Liman \cite[Theorem 1]{wl} proved:\\
\begin{thmx}\label{tA}
   \textit{If} $P\in \mathscr P_{n}$, \textit{then for every } $R \geq 1$ \textit{and} $p \geq 1$,
  \begin{equation}\label{tae}
  \left\|B[P\circ\sigma](z)\right\|_p\leq R^{n}|\Lambda_n|\left\|P(z)\right\|_{p}, 
  \end{equation}
  \textit{where} $B\in \mathcal{B}_{n}$, $\sigma(z)=Rz$ \textit{and} $\Lambda_n$ \textit{is defined by} \eqref{phi}.\\
\end{thmx}  
While seeking the analogue of \eqref{eq12} in $L_p$ norm, they \cite[Theorem 2]{wl}  have made an incomplete attempt by claiming to have proved the following result:\\

\begin{thmx}\label{tB}
\textit{If} $P\in \mathscr {P}_{n}$, \textit{and} $P(z)$ \textit{does not vanish for} $|z|\leq 1,$ \textit{then for each} $p\geq 1$, $R \geq 1$,
\begin{equation} \label{tbe}
\left\|B[P\circ\sigma](z)\right\|_p \leq \frac{R^{n}|\Lambda_n|+|\lambda_{0}|}{\left\|1+z\right\|_{p}}\left\|P(z)\right\|_{p},
\end{equation}
\textit{where} $B\in \mathcal{B}_{n}$, $\sigma(z)=Rz$ \textit{and} $\Lambda_n$ \textit{is defined by} \eqref{phi}.\\
\end{thmx}
 
\indent Unfortunately the proof of inequality \eqref{tbe} and other related results including the key lemma \cite[Lemma 4]{wl} given by Shah and Liman is not correct. The reason being that the authors in \cite{wl} deduce \cite[line 10 from line 7 on page 84, line 19 on page 85 from Lemma 3, line 16 from line 14 on page 86]{wl}  by using the argument that if $P^{\star}(z):= z^{n}\overline{P(1/\overline{z})}$, then for $\sigma(z)=Rz$, $R \geq 1$ and $|z|=1,$
$$|B[P^{\star}\circ\sigma](z)|=|B[(P^{\star}\circ\sigma)^{\star}](z)|,$$
which is not true, in general, for every $R\geq 1$ and $|z|=1$. To  see this, let
\[P(z)=a_{n}z^{n}+\cdots+a_{k}z^{k}+\cdots+a_{1}z+a_{0}\] be an arbitrary polynomial of degree $n$, then 
\[P^{\star}(z):=z^{n}\overline{P(1/\overline{z})}= \bar{a_{0}}z^{n}+\bar{a_{1}}z^{n-1}+\cdots+\bar{a_{k}}z^{n-k}+\cdots+\bar{a_{n}}.\]
Now with $\omega_{1}:=\lambda_{1}n/2$ and $\omega_{2}:=\lambda_{2}n^{2}/8$, we have
\[B[P^{\star}\circ\sigma](z)= \sum_{k=0}^{n}\left(\lambda_{0}+\omega_{1}(n-k)+\omega_{2}(n-k)(n-k-1)\right)\bar{a_{k}}z^{n-k}R^{n-k},\]
and in particular for $|z|=1$, we get
\[B[P^{\star}\circ\sigma](z)= R^{n}z^{n}\sum_{k=0}^{n}\left(\lambda_{0}+\omega_{1}(n-k)+\omega_{2}(n-k)(n-k-1)\right)\overline{{a_{k}}\left(\frac{z}{R}\right)^{k}},\]
whence
\[|B[P^{\star}\circ\sigma](z)|= R^{n}\left|\sum_{k=0}^{n}\overline{\left(\lambda_{0}+\omega_{1}(n-k)+\omega_{2}(n-k)(n-k-1)\right)}a_{k}\left(\frac{z}{R}\right)^{k}\right|.\]
But
\[|B[(P^{\star}\circ\sigma)^{\star}](z)|=R^{n}\left|\sum_{k=0}^{n}\left(\lambda_{0}+\omega_{1}k+\omega_{2}k(k-1)\right)a_{k}\left(\frac{z}{R}\right)^{k}\right|,\]
so the asserted identity does not hold in general for every $R\geq 1$ and $|z|=1$ as e.g. the immediate counterexample of $P(z):=z^n$ demonstrates in view of $P^{\star}(z)=1$, $|B[P^{\star}\circ\sigma](z)|=|\lambda_0|$ and\\
$$|B[(P^{\star}\circ\sigma)^{\star}](z)|=|\lambda_{0}+\lambda_{1}(n^{2}/2) +\lambda_{2}n^{3}(n-1)/8|\,\,\,\,(|z|=1).$$\\
As claimed by \cite{wl}, Theorem \ref{tB} is sharp has remained to be verified. In fact, this claim is also wrong.\\

\indent The main aim of this paper is to establish  $L_p$-mean extensions of the inequality \eqref{eq12} for $0\leq p<\infty$ and present correct proofs of the results mentioned in \cite{wl}.
 In this direction, we present the following interesting compact generalization of Theorem \ref{tB} which yields $L_p$ mean extension of the inequality \eqref{eq10} for $0\leq p <\infty$ which  among other things includes a correct proof of inequality \eqref{tbe} for $1\leq p <\infty$ as a special case.

\begin{theorem}\label{t2}
If $P\in \mathscr{P}_n$ and $P(z)$ does not vanish for $|z|<1,$ then for $\alpha,\delta\in \mathbb{C}$ with $|\alpha|\leq 1,|\delta|\leq 1,$ $0\leq p<\infty$ and $R>1,$
\begin{align}\label{t2e2}\nonumber
\Big\|B[P\circ\sigma](e^{i\theta})&-\alpha B[P](e^{i\theta})+\delta\Big\{\dfrac{(|R^n-\alpha||\Lambda_n|-|1-\alpha||\lambda_0|)m}{2}\Big\}\Big\|_p\\&\leq \dfrac{\left\|(R^n-\alpha)\Lambda_nz+(1-\alpha)\lambda_0\right\|_p}{\|1+z\|_p}\left\|P(z)\right\|_p,
\end{align}
where
$m=Min_{|z|=1}|P(z)|,$ $B\in \mathcal{B}_n,$ $\sigma(z)=Rz$ and $\Lambda_n$ is defined by \eqref{phi}. The result is best possible and equality in \eqref{t2e2} holds for $P(z)=az^n+b,$ $|a|=|b|=1.$
\end{theorem} 
Setting $\delta=0$ in \eqref{t2e2}, we get the following result.
\begin{corollary}\label{c1}
If $P\in \mathscr{P}_n$ and $P(z)$ does not vanish for $|z|<1,$ then for $\alpha,\delta\in \mathbb{C}$ with $|\alpha|\leq 1,|\delta|\leq 1,$ $0\leq p<\infty$ and $R>1,$
\begin{align}\label{ce1}
\Big\|B[P\circ\sigma](e^{i\theta})-\alpha B[P](e^{i\theta})\Big\|_p\leq \dfrac{\left\|(R^n-\alpha)\Lambda_nz+(1-\alpha)\lambda_0\right\|_p}{\|1+z\|_p}\left\|P(z)\right\|_p,
\end{align}
 $B\in \mathcal{B}_n,$ $\sigma(z)=Rz$ and $\Lambda_n$ is defined by \eqref{phi}. The result is best possible and equality in \eqref{t2e2} holds for $P(z)=az^n+b,$ $|a|=|b|=1.$
\end{corollary}
If we take $\alpha=0$ in \eqref{ce1}, we get the following result which is the generalization of Theorem \ref{tB} for $p\geq 1$ and also extends it for $0\leq p<1.$

\begin{corollary}\label{c2}
If $P\in \mathscr{P}_n$ and $P(z)$ does not vanish for $|z|<1,$ then for $0\leq p<\infty$ and $R> 1,$
\begin{equation}\label{ce2}
\left\|B[P\circ\sigma](z)\right\|_p\leq \dfrac{\left\|R^n\Lambda_nz+\lambda_0\right\|_p}{\|1+z\|_p}\left\|P(z)\right\|_p,
\end{equation}
 $B\in \mathcal{B}_n,$ $\sigma(z)=Rz$ and $\Lambda_n$ is defined by \eqref{phi}. The result is sharp as shown by $P(z)=az^n+b,$ $|a|=|b|=1.$
\end{corollary}
 By triangle inequality, the following result follows immediately from Corollary \ref{c2}.
 \begin{corollary}\label{c3}
If $P\in \mathscr{P}_n$ and $P(z)$ does not vanish for $|z|<1,$ then for $0\leq p<\infty$ and $R> 1,$
\begin{equation}\label{ce3}
\left\|B[P\circ\sigma](z)\right\|_p\leq \dfrac{R^n\left|\Lambda_n|+|\lambda_0\right|}{\|1+z\|_p}\left\|P(z)\right\|_p,
\end{equation}
 $B\in \mathcal{B}_n,$ $\sigma(z)=Rz$ and $\Lambda_n$ is defined by \eqref{phi}.
 \end{corollary}
 \begin{remark}
 \textnormal{Corollary \ref{c3} establishes a correct proof of a result due to Shah and Liman \cite[Theorem 3]{wl} for $p\geq 1$ and also extends it for $0\leq p<1$ as well.}
 \end{remark}
\begin{remark}
 \textnormal{If we choose $\lambda_{0}=0=\lambda_{2}$ in \eqref{ce2}, we get for $ R > 1$ and $0\leq p <\infty,$
\begin{equation*}
\left\|P^{\prime}(Rz)\right\|_p\leq \frac{nR^{n-1}}{\left\|1+z\right\|_{p}}\left\|P(z)\right\|_p, 
\end{equation*}
which, in particular, yields inequality \eqref{eq7}. Next if we take $\lambda_{1}=0=\lambda_{2}$ in \eqref{ce2}, we get inequality \eqref{eq8}. Inequality \eqref{eq10} can be obtained  from corollary \ref{c2} by letting $p\rightarrow \infty$ in \eqref{t2e2}.}\\   
\end{remark}
Taking $\alpha=0$ in \eqref{t2e2}, we get the following result.
\begin{corollary}\label{cc}
If $P\in \mathscr{P}_n$ and $P(z)$ does not vanish for $|z|<1,$ then for $\delta\in \mathbb{C}$ with $|\delta|\leq 1,$ $0\leq p<\infty$ and $R>1,$
\begin{align}\label{cce}
\Big\|B[P\circ\sigma](z)+\delta\Big\{\dfrac{(R^n|\Lambda_n|-|\lambda_0|)m}{2}\Big\}\Big\|_p\leq \dfrac{\left\|R^n\Lambda_nz+\lambda_0\right\|_p}{\|1+z\|_p}\left\|P(z)\right\|_p,
\end{align}
where
$m=Min_{|z|=1}|P(z)|,$ $B\in \mathcal{B}_n,$ $\sigma(z)=Rz$ and $\Lambda_n$ is defined by \eqref{phi}. The result is best possible and equality in \eqref{cce} holds for $P(z)=az^n+b,$ $|a|=|b|=1.$
\end{corollary}
The following corollary immediately follows from Theorem \ref{t2} by letting $p\rightarrow\infty$ in \eqref{t2e2} and choosing the argument of $\delta$ suitably with $|\delta|=1.$ 
\begin{corollary}\label{c4}
If $P\in \mathscr{P}_n$ and $P(z)$ does not vanish for $|z|<1,$ then for $\alpha\in\mathbb{C}$ with $|\alpha|\leq 1,$ $R>1,$
\begin{align} \label{ce4}\nonumber
\Big\|B[P\circ\sigma](z)-\alpha B[P](z)\Big\|_{\infty}\leq& \dfrac{\left|R^n-\alpha\right||\Lambda_n|+\left|(1-\alpha)\lambda_0\right|}{2}\left\|P(z)\right\|_{\infty}\\&-\dfrac{\left|R^n-\alpha\right||\Lambda_n|-\left|(1-\alpha)\lambda_0\right|}{2}m,
\end{align}
 where $m=Min_{|z|=1}|P(z)|,$ $B\in \mathcal{B}_n,$ $\sigma(z)=Rz$ and $\Lambda_n$ is defined by \eqref{phi}.
\end{corollary}
\begin{center}
\section{\bf Lemma}
\end{center}
For the proofs of this theorem, we need the following lemmas. The first lemma follows from Corollary $18.3$ of \cite[p. 86]{mm}.
\begin{lemma}\label{l1}
If $ P\in\mathscr{P}_n $ and $P(z)$ has all zeros in $|z|\leq 1,$  then all the zeros of $B[P](z)$ also lie in $|z|\leq 1.$ 
\end{lemma}
 \begin{lemma}\label{l2}
   If $P\in\mathscr{P}_n $ and $ P(z) $ have  all its zeros in $\left|z\right|\leq 1$ then for every $R>1,$  and $\left|z\right|=1$,
  \begin{equation*}
  \left|P(Rz)\right|\geq \left(\frac{R+1}{2}\right)^{n}\left|P(z)\right|.
  \end{equation*}
  \end{lemma}
  \begin{proof}
  Since all the zeros of $P(z)$ lie in $\left|z\right|\leq 1$,  we write
  \[P(z)=C\prod_{j=1}^{n}\left(z-r_{j}e^{i\theta_{j}}\right),\]
   where $r_{j}\leq 1$. Now for $0\leq \theta <2\pi$, $R >  1$, we have
   \begin{align*}
  \left|\frac{Re^{i\theta}-r_{j}e^{i\theta_{j}}}{e^{i\theta}-r_{j}e^{i\theta_{j}}}\right|&= \left\{\frac{R^{2}+r_{j}^{2}-2Rr_{j}\cos(\theta-\theta_{j})}{1+r_{j}^{2}-2r_{j}\cos(\theta-\theta_{j})}\right\}^{1/2},\\
   &\geq \left\{\frac{R+r_{j}}{1+r_{j}}\right\},\\
   &\geq \left\{\frac{R+1}{2}\right\}, \textrm{for}\,\,\,j = 1,2,\cdots,n.
  \end{align*}
  Hence
  \begin{align*}
  \left|\frac{P(Re^{i\theta})}{P(e^{i\theta})}\right|&=\prod_{j=1}^{n}\left|\frac{Re^{i\theta}-r_{j}e^{i\theta_{j}}}{e^{i\theta}-r_{j}e^{i\theta_{j}}}\right|,\\
  &\geq \prod_{j=1}^{n}\left(\frac{R+1}{2}\right),\\
  &=\left(\frac{R+1}{2}\right)^{n},
  \end{align*}
  for $0\leq \theta <2\pi$. This implies for $|z|=1$ and $R>1$,
  $$ \left| P(Rz)\right|\geq \left(\frac{R+1}{2}\right)^{n}\left|P(z)\right|,$$
  which completes the proof of Lemma \ref{l2}.
\end{proof}
\begin{lemma}\label{l2'}
If $P\in\mathscr P_n$ and $P(z)$ has all its zeros in $|z|\leq 1,$ then for every real or complex number $\alpha$ with $|\alpha|\leq 1,$ $R>1$ and $|z|\geq 1,$
\begin{align}\label{le2'}
|B[P\circ\sigma](z)-\alpha B[P](z)|\geq |R^n-\alpha ||\Lambda_n| |z|^n m,
\end{align}
where $m=\underset{|z|=1}{Min}|P(z)|,$ $\sigma(z)=Rz$ and $\Lambda_n$ is given by \eqref{phi}.
\end{lemma}
\begin{proof}
By hypothesis, all the zeros of $P(z)$ lie in $|z|\leq 1$ and 
$$  m |z|^n\leq |P(z)|\,\,\,\,\textnormal{for}\,\,\,\,|z|=1. $$ 
We first show that the polynomial $g(z)=P(z)-\beta mz^n$ has all its zeros in $|z|\leq 1$ for every real or complex number $\beta$ with $|\beta|<1.$ This is obvious if $m=0,$ that is if $P(z)$ has a zero on $|z|=1.$ Henceforth, we assume $P(z)$ has all its zeros in $|z|<1,$ then $m>0$ and it follows by Rouche's theorem that the polynomial $g(z) $ has all its zeros in $|z|<1$ for every real or complex number $\beta$ with $|\beta|<1.$ Applying Lemma \ref{l2} to the polynomial $g(z),$ we deduce
\begin{align*}
|g(Rz)|\geq \left(\dfrac{R+1}{2}\right)^n|g(z)|\,\,\,\,\textnormal{for}\,\,\,\,|z|=1\,\,\,\,R>1.
\end{align*}   
Since $R>1,$ therefore $\frac{R+1}{2}>1,$ this gives
\begin{align}\label{p1}
|g(Rz)|>|g(z)|\,\,\,\,\textnormal{for}\,\,\,\,|z|=1\,\,\,\,R>1.
\end{align}
Since all the zeros of $G(Rz)$ lie in $|z|<1/R<1,$ by Rouche's theorem again it follows from \eqref{p1} that all the zeros of polynomial 
\begin{align*}
H(z)=g(Rz)-\alpha g(z)=P(Rz)-\alpha P(z)-\beta(R^n-\alpha)z^n m
\end{align*}
lie in $|z|<1,$ for every $\alpha,\beta$ with $|\alpha|\leq 1,|\beta|<1$ and $R>1.$ Applying Lemma \ref{l1} to $H(z)$ and noting that $B$ is a linear operator, it follows that all the zeros of polynomial 
\begin{align}\label{p1'}\nonumber
B[H](z)&=B[g\circ\sigma](z)-\alpha B[g](z)\\&=\left\{B[P\circ\sigma](z)-\alpha B[P](z)    \right\}-\beta(R^n-\alpha)m B[z^n]
\end{align} 
lie in $|z|<1.$ This gives
\begin{align}\label{p2}
\left|B[P\circ\sigma](z)-\alpha B[P](z)\right|\geq |R^n-\alpha||\Lambda_n||z|^nm\,\,\,\,\,\textnormal{for}\,\,\,\,\,|z|\geq 1.
\end{align}
If \eqref{p2} is not true, then there is point $w$ with $|w|\geq 1$ such that 
\begin{align}
\left|B[P\circ\sigma](w)-\alpha B[P](w)\right|< |R^n-\alpha||\Lambda_n||w|^nm.
\end{align}
We choose
$$ \beta=\dfrac{B[P\circ\sigma](w)-\alpha B[P](w)}{(R^n-\alpha)\Lambda_nw^n m},   $$
then clearly $|\beta|<1$ and with this choice of $\beta,$ from \eqref{p1'}, we get $B[H](w)=0$ with $|w|\geq 1.$ This is clearly a contradiction to the fact that all the zeros of $H(z)$ lie in $|z|<1.$ Thus for every real or complex $\alpha$ with $|\alpha|\leq $  1,
$$\left|B[P\circ\sigma](z)-\alpha B[P](z)\right|\geq |R^n-\alpha||\Lambda_n||z|^nm$$
for $|z|\geq 1$ and $R>1.$
\end{proof} 
  \begin{lemma}\label{l3}
   If $ P\in\mathscr{P}_n $ and $ P(z) $ has no zero in $\left|z\right|<1,$ then for every $ \alpha\in\mathbb{C} $ with $ |\alpha|\leq 1,$ $ R>1 $ and $ |z|\geq 1 $,
  \begin{equation}\label{le3}
  \left|B[P\circ\sigma](z)-\alpha B[P](z)\right|\leq \left|B[P^{\star}\circ\sigma](z)-\alpha B[P^{\star}](z)\right|,
  \end{equation}
    where $P^{\star}(z)=z^{n}\overline{P(1/\overline{z})} $ and $\sigma(z)=Rz.$
    \end{lemma}
   \begin{proof}
    Since the polynomial $P(z)$ has all its zeros in $|z|\geq 1,$ therefore, for every real or complex number $\lambda$ with $|\lambda|>1,$ the polynomial $f(z)=P(z)-\lambda P^{\star}(z),$ where $P^{\star}(z)=z^{n}\overline{P(1/\overline{z})},$ has all zeros in $|z|\leq 1.$ Applying Lemma \ref{l2} to the polynomial $f(z),$ we obtain for every $R> 1$ and $0\leq \theta<2\pi,$
   \begin{equation}\label{l3p1}
   |f(Re^{i\theta})|\geq \left(\dfrac{R+1}{2}\right)^n|f(e^{i\theta})|.
   \end{equation} 
   Since $f(Re^{i\theta})\neq 0$ for every $R>1,$ $0\leq \theta<2\pi$ and $R+1>2,$ it follows from \eqref{l3p1} that
   \begin{equation*}
   |f(Re^{i\theta})|> \left(\dfrac{R+1}{2}\right)^n|f(Re^{i\theta})|\geq |f(e^{i\theta})|,
   \end{equation*}
   for every $R>1$ and $0\leq \theta<2\pi.$ This gives
   \begin{equation*}
   |f(z)|<|f(Rz)|\,\,\,\textnormal{for}\,\,\,\,|z|=1,\,\,\,\,\textnormal{and}\,\,\,R> 1.
   \end{equation*}
   Using Rouche's theorem and noting that all the zeros of $f(Rz)$ lie in $|z|\leq 1/R<1,$ we conclude that the polynomial
   $$ T(z)=f(Rz)-\alpha f(z)=\left\{P(Rz)-\alpha P(z)\right\}-\lambda\left\{P^{\star}(Rz)-\alpha P^{\star}(z)\right\}  $$
   has all its zeros in $|z|<1$ for every real or complex $\alpha$ with $|\alpha|\geq 1$ and $R> 1.$ \\
   Applying Lemma \ref{l1} to polynomial $T(z)$ and noting that $B$ is a linear operator, it follows that all the zeros of polynomial
    \begin{align*}
B[T](z)&=B[f\circ\sigma](z)-\alpha B[f](z)\\&=\left\{B[P\circ\sigma](z)-\alpha B[P](z)\right\}-\lambda\left\{B[P^{\star}\circ\sigma](z)-\alpha B[P^{\star}](z)\right\} 
    \end{align*}
    lie in $|z|<1$ where $\sigma(z)=Rz.$ This implies
   \begin{equation}\label{l3p2}
   |B[P\circ\sigma](z)-\alpha B[P](z)|\leq |B[P^{\star}\circ\sigma](z)-\alpha B[P^{\star}](z)|
   \end{equation}
   for $|z|\geq 1$ and $R>1.$ If inequality \eqref{l3p2} is not true, then there exits a point $z=z_0$ with $|z_0|\geq 1$ such that
  \begin{equation}\label{l3p}
     |B[P\circ\sigma](z_0)-\alpha B[P](z_0)|> |B[P^{\star}\circ\sigma](z_0)-\alpha B[P^{\star}](z_0)|.
     \end{equation} 
     But all the zeros of $P^{\star}(Rz)$ lie in $|z|<1/R<1,$ therefore, it follows (as in case of $f(z)$) that all the zeros of $P^{\star}(Rz)-\alpha P^{\star}(z)$ lie in $|z|<1.$ Hence, by Lemma \ref{l1}, we have
     $$B[P^{\star}\circ\sigma](z_0)-\alpha B[P^{\star}](z_0)\neq 0 .$$
     We take
     $$ \lambda=\dfrac{B[P\circ\sigma](z_0)-\alpha B[P](z_0)}{B[P^{\star}\circ\sigma](z_0)-\alpha B[P^{\star}](z_0)} , $$
     then $\lambda$ is well defined real or complex number with $|\lambda|>1$ and with this choice of $\lambda,$ we obtain $B[T](z_0)=0$ where $|z_0|\geq 1.$ This contradicts the fact that all the zeros of $B[T](z)$ lie in $|z|<1.$ Thus \eqref{l3p2} holds true for $|\alpha|\leq 1$ and $R>1.$
   \end{proof}
   \begin{lemma}\label{l3'}
      If $ P\in\mathscr{P}_n $ and $ P(z) $ has no zero in $\left|z\right|<1,$ then for every $ \alpha\in\mathbb{C} $ with $ |\alpha|\leq 1,$ $ R>1 $ and $ |z|\geq 1 $,
     \begin{align}\label{le3'}\nonumber
     \big|B[P\circ\sigma]&(z)-\alpha B[P](z)\big|\\\leq &\left|B[P^{\star}\circ\sigma](z)-\alpha B[P^{\star}](z)\right|-(|R^n-\alpha||\Lambda_n|-|1-\alpha||\lambda_0|)m,
     \end{align}
       where $P^{\star}(z)=z^{n}\overline{P(1/\overline{z})} ,$ $m=\underset{|z|=1}{Min}|P(z)|$ and $\sigma(z)=Rz.$
       \end{lemma}
       \begin{proof}
By hypothesis $P(z)$ has all its zeros in $|z|\geq 1$ and 
\begin{align}\label{q1}
m\leq |P(z)|\,\,\,\textnormal{for}\,\,\,\,|z|=1.
\end{align}
We show $F(z)=P(z)+\lambda m$ does not vanish in $|z|<1$ for every $\lambda$ with $|\lambda|<1.$ This is obvious if $m=0$ that is, if $P(z)$ has a zero on $|z|=1.$ So we assume all the zeros of $P(z)$ lie in $|z|>1,$ then $m>0$ and by the maximum modulus principle, it follows from \eqref{q1},
\begin{align}\label{q2}
m< |P(z)|\,\,\,\textnormal{for}\,\,\,|z|<1.
\end{align} 
Now if $F(z)=P(z)+\lambda m=0$ for some $z_0$ with $|z_0|<1,$ then
\begin{align*}
P(z_0)+\lambda m=0
\end{align*} 
This implies
\begin{align}
|P(z_0)|= |\lambda|m\leq m,\,\,\,\textnormal{for}\,\,\,|z_0|<1
\end{align} 
which is clearly contradiction to \eqref{q2}. Thus the polynomial $F(z)$ does not vanish in $|z|<1$  for every $\lambda$ with $|\lambda|<1.$ applying Lemma \ref{l3} to the polynomial $F(z),$ we get
\begin{align*}
|B[F\circ\sigma](z)-\alpha B[F](z)|\leq |B[F^{\star}\circ\sigma](z)-\alpha B[F^{\star}](z)
\end{align*} 
for $|z|=1$ and $R>1.$ Replacing $F(z)$ by $P(z)+\lambda m,$ we obtain
\begin{align}\label{q3}\nonumber
|B[P\circ\sigma](z)&-\alpha B[P](z)+\lambda(1-\alpha)\lambda_0 m|\\&\leq |B[P^{\star}\circ\sigma](z)-\alpha B[P^{\star}](z)+\bar{\lambda} (R^n-\alpha)\Lambda_nz^n m|
\end{align}
 Now choosing the argument of $\lambda$ in the right hand side of \eqref{q3} such that
 \begin{align*}\nonumber
|B[P^{\star}\circ\sigma](z)&-\alpha B[P^{\star}](z)+\bar{\lambda} (R^n-\alpha)\Lambda_nz^nm|\\&=|B[P^{\star}\circ\sigma](z)-\alpha B[P^{\star}](z)|-|\lambda||R^n-\alpha||\Lambda_n|m
\end{align*}
for $|z|=1,$which is possible by Lemma \ref{l2'},we get
\begin{align*}\nonumber
|B[P^{\star}\circ\sigma](z)-\alpha B[P^{\star}](z)|&-|\lambda||1-\alpha||\lambda_0|m\\\leq& |B[P^{\star}\circ\sigma](z)-\alpha B[P^{\star}](z)|-|\lambda||R^n-\alpha||\Lambda_n|m
\end{align*}
Equivalently,
\begin{align*}
     \big|B[P\circ\sigma]&(z)-\alpha B[P](z)\big|\\\leq &\left|B[P^{\star}\circ\sigma](z)-\alpha B[P^{\star}](z)\right|-(|R^n-\alpha||\Lambda_n|-|1-\alpha||\lambda_0|)m,
\end{align*}
       \end{proof}
 Next we describe a result of Arestov \cite{va}.\\
   
  For $\delta = (\delta_{0},\delta_{1},\cdots, \delta_{n})\in\mathbb{C}^{n+1}$ and $P(z)=\sum_{j=0}^{n}a_{j}{z}^{j}\in \mathscr{P}_{n}$, we define 
   
  \[\psi_\delta P(z)=\sum_{j=0}^{n}\delta_{j} a_{j}{z}^{j}.\]
   
  The operator $\psi_\delta$ is said to be admissible if it preserves one of the following properties:\\
    (i)     \,\,  $P(z)$ has all its zeros in $\left\{z \in \mathbb{C}: |z|\leq 1\right\},$\\
   (ii)     \,$P(z)$ has all its zeros in$ \left\{z \in \mathbb{C}: |z|\geq 1\right\}.$\\
   The result of Arestov \cite{va} may now be stated as follows.\\
 
 \begin{lemma} \cite[Theorem 4]{va} \label{l4}
    Let $\phi(x) = \rho(log x)$ where $\rho$ is a convex non decreasing function on $\mathbb{R}.$ Then for all $P\in \mathscr{P}_{n}$ and each admissible operator $\psi_\delta$,
   \[\int_{0}^{2\pi} \phi(|\psi_\delta P(e^{i\theta})|)d\theta \leq \int_{0}^{2\pi}\phi(C(\delta,n)|P(e^{i\theta})|)d\theta, \]
  
  where $C( \delta,n)= max(|\delta_{0}|,|\delta_{n}|).$\\
  \end{lemma}
  In particular, Lemma \ref{l4} applies with $\phi:x\rightarrow x^{p}$ for every $p\in(0,\infty)$. Therefore, we have 
  \begin{equation}\label{le4'}
  \left\{\int_{0}^{2\pi}(|\psi_\delta P(e^{i\theta})|^{p})d\theta\right\}^{1/p} \leq C(\delta,n)\left\{\int_{0}^{2\pi}|P(e^{i\theta})|^{p}d\theta \right\}^{1/p}.
  \end{equation}
  We use \eqref{le4'} to prove the following interesting result.\\
  \begin{lemma}\label{l5}
  If $P\in \mathscr{P}_{n}$ and $P(z)$ does not vanish in $|z|<1,$ then for every $p > 0$, $R > 1$ and for  $\gamma$ real, $0 \leq\gamma <2\pi$,
\begin{align}\nonumber\label{le5}
  \int_{0}^{2\pi}\Big|\Big\{B[P\circ\sigma](e^{i\theta})&-\alpha B[P](e^{i\theta})\Big\}\mbox{}\, e^{i\gamma}\\\nonumber &+\left\{B[P^{\star}\circ\sigma]^{\star}(e^{i\theta})-\bar{\alpha} B[P^{\star}]^{\star}(e^{i\theta})\right\}\Big|^{p}d\theta\\\leq \Big|(R^{n}-\alpha )&\Lambda_n e^{i\gamma}+(1-\bar{\alpha})\bar{\lambda_{0}}\Big|^{p}\int_{0}^{2\pi}\left|P(e^{i\theta})\right|^{p}d\theta,
\end{align}
  where $B\in \mathcal{B}_{n}$, $\sigma(z):=Rz$, $B[P^{\star}\circ\sigma]^{\star}(z):=(B[P^{\star}\circ\sigma](z))^{\star}$ and $\Lambda_n$ is defined by \eqref{phi}.\\
  \end{lemma}
  \begin{proof}
   Since $P\in \mathscr P_{n}$ and $P^{\star}(z)= z^{n}\overline{P(1/\bar{z})}$, by Lemma \ref{l3} , we have for $|z|\geq 1,$
  \begin{equation}\label{l5p1}
  \left|B[P\circ\sigma](z)-\alpha B[P](z)\right|\leq \left|B[P^{\star}\circ\sigma](z)-\alpha B[P^{\star}](z)\right|,
  \end{equation}
  Also, since $P^{\star}(Rz)-\alpha P^{\star}(z) = R^{n}z^{n}\overline{P(1/R\bar{z})}-\alpha z^{n}\overline{P(1/\bar{z})}$, 
  \begin{align*}
  B[P^{\star}\circ\sigma](z)&-\alpha B[P^{\star}](z)=\lambda_{0}\Big\{ R^{n}z^{n}\overline{P(1/R\bar{z})}-\alpha z^{n}\overline{P(1/\bar{z})}\Big\}\\
  +\lambda_{1}&\left(\frac{nz}{2}\right)\Big\{\left(nR^{n}z^{n-1}\overline{P(1/R\bar{z})}-R^{n-1}z^{n-2}\overline{P^{\prime}(1/R\bar{z})}\right)\\&-\alpha\left(nz^{n-1}\overline{P(1/\bar{z})}-z^{n-2}\overline{P^{\prime}(1/\bar{z})}\right)\Big\}\\
  +\frac{\lambda_{2}}{2!}&\left(\frac{nz}{2}\right)^{2}\Big\{\Big(n(n-1)R^{n}z^{n-2}\overline{P(1/R\bar{z})}
  \\&-2(n-1)R^{n-1}z^{n-3}\overline{P^{\prime}(1/R\bar{z})}+R^{n-2}z^{n-4}\overline{P^{\prime\prime}(1/R\bar{z})}\Big)\\-\alpha \Big(n&(n-1)z^{n-2}\overline{P(1/\bar{z})}
    -2(n-1)z^{n-3}\overline{P^{\prime}(1/\bar{z})}\\&+r^{n-2}z^{n-4}\overline{P^{\prime\prime}(1/\bar{z})}\Big) \Big\}
  \end{align*}
  and therefore,
  \begin{align}\label{l5p2}\nonumber
   B[P^{\star}\circ\sigma]^{\star}(z)&-\bar{\alpha} B[P^{\star}]^{\star}(z)=\Big(B[P^{\star}\circ\sigma](z)-\alpha B[P^{\star}](z)\Big)^{\star}\\\nonumber 
  =\Big(\bar{\lambda_{0}}+&\bar{\lambda_{1}}\frac{n^{2}}{2}+\bar{\lambda_{2}}\frac{n^{3}(n-1)}{8}\Big)\Big\{R^{n}P(z/R)-\bar{\alpha} P(z)\Big\}\\\nonumber
  -\Big(\bar{\lambda_{1}}&\frac{n}{2}+\bar{\lambda_{2}}\frac{n^{2}(n-1)}{4}\Big)\Big\{ R^{n-1}zP^{\prime}(z/R)-\bar{\alpha} zP^{\prime}(z)\Big\}\\ +\bar{\lambda_{2}}&\frac{n^{2}}{8}\Big\{ R^{n-2}z^{2}P^{\prime\prime}(z/R)-\bar{\alpha} z^{2}P^{\prime\prime}(z)\Big\}.
  \end{align}
  Also,\\
  $$| B[P^{\star}\circ\sigma](z)-\alpha B[P^{\star}](z)|=|B[P^{\star}\circ\sigma]^{\star}(z)-\bar{\alpha} B[P^{\star}]^{\star}(z)|\,\,\,\hbox{}\,\,\,\textrm{for} \,\,\,|z|=1.$$
  Using this in \eqref{l5p1}, we get 
 $$| B[P\circ\sigma](z)-\alpha B[P](z)|\leq |B[P^{\star}\circ\sigma]^{\star}(z)-\bar{\alpha} B[P^{\star}]^{\star}(z)|\,\,\,\hbox{}\,\,\,\textrm{for} \,\,\,|z|=1.$$
  
  As in Lemma \ref{l3}, the polynomial $P^{\star}\circ\sigma(z)-\alpha P^{\star}(z),$ has all its zeros in $|z|<1$ and by Lemma \ref{l1}, $B[P^{\star}\circ\sigma](z)-\alpha B[P^{\star}](z),$ also has all its zero in $|z|<1,$ therefore, $B[P^{\star}\circ\sigma]^{\star}(z)-\bar{\alpha} B[P^{\star}]^{\star}(z)$ has all its zeros in $|z|\geq 1.$ Hence by the maximum modulus principle, 
  \begin{equation}\label{l5p3}
  | B[P\circ\sigma](z)-\alpha B[P](z)|< |B[P^{\star}\circ\sigma]^{\star}(z)-\bar{\alpha} B[P^{\star}]^{\star}(z)|\,\,\,\hbox{}\,\,\,\textrm{for} \,\,\,|z|<1
  \end{equation}
  A direct application of Rouche's theorem shows that with $P(z)=a_{n}z^{n}+\cdots+a_{0},$
  \begin{align*}
  \psi_\delta P(z)=& \Big\{B[P\circ\sigma](z)-\alpha B[P](z)\Big\}e^{i\gamma}+B[P^{\star}\circ\sigma]^{\star}(z)-\bar{\alpha} B[P^{\star}]^{\star}(z),\\
  =& \left\{(R^{n}-\alpha )\left(\lambda_{0}+\lambda_{1}\frac{n^{2}}{2}+\lambda_{2}\frac{n^{3}(n-1)}{8}\right)e^{i\gamma}+(1-\bar{\alpha})\bar{\lambda_{0}}\right\}a_{n}z^{n}\\
  &+\cdots+\left\{(R^{n}-\bar{\alpha} )\left(\bar{\lambda_{0}}+\bar{\lambda_{1}}\frac{n^{2}}{2}+\bar{\lambda_{2}}\frac{n^{3}(n-1)}{8}\right)+e^{i\gamma}(1-\alpha)\lambda_{0}\right\}a_{0},
  \end{align*}
  has all its zeros in $|z|\geq 1$, for every real $\gamma,$ $0\leq \gamma\leq 2\pi.$  Therefore, $\psi_\delta$ is an admissible operator. Applying \eqref{le4'} of Lemma \ref{l4}, the desired result follows immediately for each $p > 0$.\\
  \end{proof}
 
 We also need the following lemma \cite{ar97}.
 \begin{lemma}\label{abc}
 If $A,B,C$ are non-negative real numbers such that $B+C\leq A,$ then for each real number $\gamma,$
 $$ |(A-C)e^{i\gamma}+(B+C)|\leq |Ae^{i\gamma}+B|.    $$
 \end{lemma}
  
  \section{\bf Proof of Theorem}
  \begin{proof}[\textnormal{\bf Proof of Theorem \ref{t2}}]
  By hypothesis $P(z)$ does not vanish in $|z|<1,$ therefore by Lemma \ref{l3'}, we have
  \begin{align}\label{r1}\nonumber
     \big|B[P\circ\sigma]&(z)-\alpha B[P](z)\big|\\\leq &\left|B[P^{\star}\circ\sigma](z)-\alpha B[P^{\star}](z)\right|-(|R^n-\alpha||\Lambda_n|-|1-\alpha||\lambda_0|)m,
  \end{align}
   for $|z|=1,$ $|\alpha|\leq 1$ and $R>1$ where $P^{\star}(z)=z^n\overline{P(1/\overline{z})}.$
 Since $B[P^{\star}\circ\sigma]^{\star}(z)-\bar{\alpha}B[P^{\star}]^{\star}(z)$ is the conjugate of $B[P^{\star}\circ\sigma](z)-\alpha B[P^{\star}](z)$ and $$|B[P^{\star}\circ\sigma]^{\star}(z)-\bar{\alpha}B[P^{\star}]^{\star}(z)|=|B[P^{\star}\circ\sigma](z)-\alpha B[P^{\star}](z)|\,\,\,\,\textnormal{for}\,\,\,\,\,|z|=1.$$
 Thus \eqref{r1} can be written as 
 \begin{align}\label{r2}\nonumber
 |B[P\circ\sigma](z)&-\alpha B[P](z)|+\dfrac{(|R^n-\alpha||\Lambda_n|-|1-\alpha||\lambda_0|)m}{2}\\\leq&\left|B[P^{\star}\circ\sigma]^{\star}(z)-\bar{\alpha}B[P^{\star}]^{\star}(z)\right|-\dfrac{(|R^n-\alpha||\Lambda_n|-|1-\alpha||\lambda_0|)m}{2}\,\,\,\,\,\mbox{}\textrm{for}\,\,\,\,\,|z|=1
 \end{align}
 Taking 
 $$A= \left|B[P^{\star}\circ\sigma]^{\star}(z)-\bar{\alpha}B[P^{\star}]^{\star}(z)\right|,\,\,\,B=\left|B[P\circ\sigma](z)-\alpha B[P](z)\right|$$
 and 
 $$  C=\dfrac{(|R^n-\alpha||\Lambda_n|-|1-\alpha||\lambda_0|)m}{2} $$
 in Lemma \ref{abc} and noting by \eqref{r2} that 
 $$ B+C\leq A-C\leq A,   $$
 we get for every real $\gamma$,
 \begin{align*}
 \Big|\Big\{\big|B[P^{\star}&\circ\sigma]^{\star}(e^{i\theta})-\bar{\alpha}B[P^{\star}]^{\star}(e^{i\theta})\big|-\dfrac{(|R^n-\alpha||\Lambda_n|-|1-\alpha||\lambda_0|)m}{2}\Big\}e^{i\gamma}\\+\Big\{&\big|B[P\circ\sigma](e^{i\theta})-\alpha B[P](e^{i\theta})\big|+\dfrac{(|R^n-\alpha||\Lambda_n|-|1-\alpha||\lambda_0|)m}{2}\Big\}\Big|\\&\leq \Big|\big|B[P^{\star}\circ\sigma]^{\star}(e^{i\theta})-\bar{\alpha}B[P^{\star}]^{\star}(e^{i\theta})\big|e^{i\gamma}+\big|B[P\circ\sigma](e^{i\theta})-\alpha B[P](e^{i\theta})\big|\Big|.
 \end{align*}
 This implies for each $p>0,$
 \begin{align}\nonumber\label{r3}
\int\limits_{0}^{2\pi} \Bigg|\Big\{\big|B[P^{\star}&\circ\sigma]^{\star}(e^{i\theta})-\bar{\alpha}B[P^{\star}]^{\star}(e^{i\theta})\big|-\dfrac{(|R^n-\alpha||\Lambda_n|-|1-\alpha||\lambda_0|)m}{2}\Big\}e^{i\gamma}\\\nonumber+\Big\{&\big|B[P\circ\sigma](e^{i\theta})-\alpha B[P](e^{i\theta})\big|+\dfrac{(|R^n-\alpha||\Lambda_n|-|1-\alpha||\lambda_0|)m}{2}\Big\}\Bigg|^pd\theta\\&\leq \int\limits_{0}^{2\pi}\Big|\big|B[P^{\star}\circ\sigma]^{\star}(e^{i\theta})-\bar{\alpha}B[P^{\star}]^{\star}(e^{i\theta})\big|e^{i\gamma}+\big|B[P\circ\sigma](e^{i\theta})-\alpha B[P](e^{i\theta})\big|\Big|^pd\theta.
 \end{align}
 Integrating both sides of \eqref{r3} with respect to $\gamma$ from $0$ to $2\pi,$ we get with the help of Lemma \ref{l5} for each $p>0,$
 \begin{align}\nonumber\label{r4}
\int\limits_{0}^{2\pi}\int\limits_{0}^{2\pi} \Bigg|&\Big\{\big|B[P^{\star}\circ\sigma]^{\star}(e^{i\theta})-\bar{\alpha}B[P^{\star}]^{\star}(e^{i\theta})\big|-\dfrac{(|R^n-\alpha||\Lambda_n|-|1-\alpha||\lambda_0|)m}{2}\Big\}e^{i\gamma}\\\nonumber+\Big\{&\big|B[P\circ\sigma](e^{i\theta})-\alpha B[P](e^{i\theta})\big|+\dfrac{(|R^n-\alpha||\Lambda_n|-|1-\alpha||\lambda_0|)m}{2}\Big\}\Bigg|^pd\theta d\gamma\\\nonumber\leq& \int\limits_{0}^{2\pi}\int\limits_{0}^{2\pi}\Bigg|\big|B[P^{\star}\circ\sigma]^{\star}(e^{i\theta})-\bar{\alpha}B[P^{\star}]^{\star}(e^{i\theta})\big|e^{i\gamma}+\big|B[P\circ\sigma](e^{i\theta})-\alpha B[P](e^{i\theta})\big|\Bigg|^pd\theta d\gamma.\\\nonumber\leq&
 \int\limits_{0}^{2\pi}\Bigg\{\int\limits_{0}^{2\pi}\Bigg|\big|B[P^{\star}\circ\sigma]^{\star}(e^{i\theta})-\bar{\alpha}B[P^{\star}]^{\star}(e^{i\theta})\big|e^{i\gamma}+\big|B[P\circ\sigma](e^{i\theta})-\alpha B[P](e^{i\theta})\big|\Bigg|^pd\gamma \Bigg\}d\theta\\\nonumber\leq&
 \int\limits_{0}^{2\pi}\Bigg\{\int\limits_{0}^{2\pi}\Bigg|\big\{B[P^{\star}\circ\sigma]^{\star}(e^{i\theta})-\bar{\alpha}B[P^{\star}]^{\star}(e^{i\theta})\big\}e^{i\gamma}+\big\{B[P\circ\sigma](e^{i\theta})-\alpha B[P](e^{i\theta})\big\}\Bigg|^p d\gamma\Bigg\}d\theta\\\nonumber\leq&
\int\limits_{0}^{2\pi}\Bigg\{\int\limits_{0}^{2\pi}\Bigg|\big\{B[P^{\star}\circ\sigma]^{\star}(e^{i\theta})-\bar{\alpha}B[P^{\star}]^{\star}(e^{i\theta})\big\}e^{i\gamma}+\big\{B[P\circ\sigma](e^{i\theta})-\alpha B[P](e^{i\theta})\big\}\Bigg|^p d\theta\Bigg\}d\gamma\\\leq& \int\limits_{0}^{2\pi}\Bigg|(R^{n}-\alpha )\Lambda_n e^{i\gamma}+(1-\bar{\alpha})\bar{\lambda_{0}}\Bigg|^{p}d\gamma\int_{0}^{2\pi}\left|P(e^{i\theta})\right|^{p}d\theta
 \end{align}
 Now it can be easily verified that for every real number $\gamma$ and $s \geq 1$,
 $$\left|s+e^{i\alpha}\right| \geq \left|1+e^{i\alpha}\right|.$$
 This implies for each $ p > 0$,
 \begin{equation}\label{r5}
 \int_{0}^{2\pi}\left|s+e^{i\gamma}\right|^{p}d\gamma \geq \int_{0}^{2\pi}\left|1+e^{i\gamma}\right|^{p}d\gamma. 
 \end{equation}
 If $\big|B[P\circ\sigma](e^{i\theta})-\alpha B[P](e^{i\theta})\big|+\dfrac{(|R^n-\alpha|-|1-\alpha||\lambda_0|)m}{2}\neq 0,$
  we take $$ s = \dfrac{\big|B[P^{\star}\circ\sigma]^{\star}(e^{i\theta})-\bar{\alpha} B[P^{\star}]^{\star}(e^{i\theta})\big|-\dfrac{(|R^n-\alpha||\Lambda_n|-|1-\alpha||\lambda_0|)m}{2}}{\big|B[P\circ\sigma](e^{i\theta})-\alpha B[P](e^{i\theta})\big|+\dfrac{(|R^n-\alpha||\Lambda_n|-|1-\alpha||\lambda_0|)m}{2}}$$, then by \eqref{r2}, $s \geq 1$ and we get with the help of \eqref{r5},\\
  \begin{align}\nonumber\label{r6}
 \int\limits_{0}^{2\pi} \Big|\Big\{&\big|B[P^{\star}\circ\sigma]^{\star}(e^{i\theta})-\bar{\alpha}B[P^{\star}]^{\star}(e^{i\theta})\big|-\dfrac{(|R^n-\alpha||\Lambda_n|-|1-\alpha||\lambda_0|)m}{2}\Big\}e^{i\gamma}\\\nonumber+\Big\{&\big|B[P\circ\sigma](e^{i\theta})-\alpha B[P](e^{i\theta})\big|+\dfrac{(|R^n-\alpha||\Lambda_n|-|1-\alpha||\lambda_0|)m}{2}\Big\}\Big|^p d\gamma\\\nonumber=&\Bigg|\big|B[P\circ\sigma](e^{i\theta})-\alpha B[P](e^{i\theta})\big|+\dfrac{(|R^n-\alpha||\Lambda_n|-|1-\alpha||\lambda_0|)m}{2}\Bigg|^p\\\nonumber&\times\int\limits_{0}^{2\pi}\left|e^{i\gamma}+\dfrac{\big|B[P^{\star}\circ\sigma]^{\star}(e^{i\theta})-\bar{\alpha}B[P^{\star}]^{\star}(e^{i\theta})\big|-\dfrac{(|R^n-\alpha||\Lambda_n|-|1-\alpha||\lambda_0|)m}{2}}{\big|B[P\circ\sigma](e^{i\theta})-\alpha B[P](e^{i\theta})\big|+\dfrac{(|R^n-\alpha||\Lambda_n|-|1-\alpha||\lambda_0|)m}{2}}\,\right|^pd\gamma\\\nonumber=&\Bigg|\big|B[P\circ\sigma](e^{i\theta})-\alpha B[P](e^{i\theta})\big|+\dfrac{(|R^n-\alpha||\Lambda_n|-|1-\alpha||\lambda_0|)m}{2}\Bigg|^p\\\nonumber&\times\int\limits_{0}^{2\pi}\left|e^{i\gamma}+\Bigg|\dfrac{\big|B[P^{\star}\circ\sigma]^{\star}(e^{i\theta})-\bar{\alpha}B[P^{\star}]^{\star}(e^{i\theta})\big|-\dfrac{(|R^n-\alpha||\Lambda_n|-|1-\alpha||\lambda_0|)m}{2}}{\big|B[P\circ\sigma](e^{i\theta})-\alpha B[P](e^{i\theta})\big|+\dfrac{(|R^n-\alpha||\Lambda_n|-|1-\alpha||\lambda_0|)m}{2}}\Bigg|^p\,\right|d\gamma\\\geq &\Bigg|\big|B[P\circ\sigma](e^{i\theta})-\alpha B[P](e^{i\theta})\big|+\dfrac{(|R^n-\alpha||\Lambda_n|-|1-\alpha||\lambda_0|)m}{2}\Bigg|^p\int\limits_{0}^{2\pi}|1+e^{i\gamma}|^pd\gamma
  \end{align}
 
 For $\big|B[P\circ\sigma](e^{i\theta})-\alpha B[P](e^{i\theta})\big|+\dfrac{(|R^n-\alpha||\Lambda_n|-|1-\alpha||\lambda_0|)m}{2} = 0$, then \eqref{r6}  is trivially true. Using this in \eqref{r4}, we conclude for every real or complex number $\alpha$ with $|\alpha|\leq 1,$ $R>1$ and $ p > 0$,
 \begin{align*}
 \int\limits_{0}^{2\pi}\Bigg|\big|B[P\circ\sigma]&(e^{i\theta})-\alpha B[P](e^{i\theta})\big|+\dfrac{(|R^n-\alpha||\Lambda_n|-|1-\alpha||\lambda_0|)m}{2}\Bigg|^pd\theta\int\limits_{0}^{2\pi}|1+e^{i\gamma}|^pd\gamma\\ 
 &\leq \int\limits_{0}^{2\pi}\Big|(R^{n}-\alpha )\Lambda_n e^{i\gamma}+(1-\bar{\alpha})\bar{\lambda_{0}}\Big|^{p}d\gamma\int_{0}^{2\pi}\left|P(e^{i\theta})\right|^{p}d\theta
 \end{align*}
 This gives for every real or complex number $\delta,\alpha$ with $|\delta|\leq 1,$ $|\alpha|\leq 1,$ $R>1$ and $\gamma$ real
 \begin{align}\nonumber\label{r7}
 \int\limits_{0}^{2\pi}\Bigg|B[P\circ\sigma]&(e^{i\theta})-\alpha B[P](e^{i\theta})+\delta\Big\{\dfrac{(|R^n-\alpha||\Lambda_n|-|1-\alpha||\lambda_0|)m}{2}\Big\}\Bigg|^pd\theta\int\limits_{0}^{2\pi}|1+e^{i\gamma}|^pd\gamma\\ 
 &\leq \int\limits_{0}^{2\pi}\Big|(R^{n}-\alpha )\Lambda_n e^{i\gamma}+(1-\bar{\alpha})\bar{\lambda_{0}}\Big|^{p}d\gamma\int_{0}^{2\pi}\left|P(e^{i\theta})\right|^{p}d\theta.
 \end{align} 
 Since
 \begin{align}\nonumber\label{r8}
\int\limits_{0}^{2\pi}\Big|(R^{n}&-\alpha )\Lambda_n e^{i\gamma}+(1-\bar{\alpha})\bar{\lambda_{0}}\Big|^{p}d\gamma\int_{0}^{2\pi}\left|P(e^{i\theta})\right|^{p}d\theta\\\nonumber
&=\int\limits_{0}^{2\pi}\Big||(R^{n}-\alpha )\Lambda_n| e^{i\gamma}+|(1-\bar{\alpha})\bar{\lambda_{0}}|\Big|^{p}d\gamma\int_{0}^{2\pi}\left|P(e^{i\theta})\right|^{p}d\theta\\\nonumber
&=\int\limits_{0}^{2\pi}\Big||(R^{n}-\alpha )\Lambda_n| e^{i\gamma}+|(1-\alpha)\lambda_{0}|\Big|^{p}d\gamma\int_{0}^{2\pi}\left|P(e^{i\theta})\right|^{p}d\theta,\\
 &=\int\limits_{0}^{2\pi}\Big|(R^{n}-\alpha )\Lambda_n e^{i\gamma}+(1-\alpha)\lambda_{0}\Big|^{p}d\gamma\int_{0}^{2\pi}\left|P(e^{i\theta})\right|^{p}d\theta,
 \end{align}
 the desired result follows immediately by combining \eqref{r7} and \eqref{r8}. This completes the proof of Theorem \ref{t2} for $p>0$. To establish this result for $p=0$, we simply let $p\rightarrow 0+$.

  \end{proof}

 \end{document}